\documentclass[11pt]{article}
\usepackage[french,english]{babel}
\usepackage{color}
\usepackage[latin1]{inputenc} 
\usepackage{amsmath} 
\usepackage{amscd} 
\usepackage{amstext} 
\usepackage{amsbsy} 
\usepackage{amsopn}
\usepackage{amsthm} 
\usepackage{amsxtra}
\usepackage{amssymb}
\usepackage{textcomp}
\usepackage{mathpazo} % rm et maths
\usepackage[scaled]{helvet} % ss
\usepackage{courier} % tt
\normalfont
\usepackage[T1]{fontenc}
\usepackage{hyperref}
\usepackage[all]{xy} 
\usepackage[active]{srcltx} 
\newcommand{\be}{\begin{enumerate}}
\newcommand{\ee}{\end{enumerate}}
\newtheorem{cor}[subsubsection]{Corollaire}
\newtheorem{souscor}[paragraph]{Corollaire}
\newtheorem{sousprop}[paragraph]{Proposition}

\newtheorem{prop}[subsubsection]{Proposition}

\newtheorem{surprop}[subsection]{Proposition}
\newtheorem{surlem}[subsection]{Lemme}

\newtheorem{lem}[subsubsection]{Lemme}
\newtheorem{souslem}[paragraph]{Lemme}

\newtheorem{sousrems}[paragraph]{Remarques}
\newcommand{\inft}{{\scriptstyle\infty}}

\newcommand{\what}{\widehat}
\newcommand{\ot}{\otimes}

\newcommand{\rig}{\rightarrow}

\newcommand{\hrig}{\hookrightarrow}
\newcommand{\sta}{\stackrel}

\newcommand{\Ne}{{\bf N}}

\newcommand{\varep}{\varepsilon}
\newcommand{\lam}{\lambda}

\newcommand{\Ga}{\Gamma}

\newcommand{\EE}{{\cal E}}

\newcommand{\OO}{{\cal O}}

\newcommand{\UU}{{\cal U}}

\newcommand{\ov}{\overline{v}}
\newcommand{\ovV}{\overline{V}}
\newcommand{\ovK}{\overline{K}}
\newcommand{\spec}{{\rm spec}\,}
\def\fil{\mathop{\rm Fil}\nolimits} 
\def\spec{\mathop{\rm Spec}\nolimits}
\def\val{\mathop{\rm val}\nolimits}
\def\fil{\mathop{\rm Fil}\nolimits} 

\begin{document}
\title{Représentations galoisiennes associées aux courbes hyperelliptiques lisses\thanks{Les deux auteurs ont bénéficié du soutien du projet CETHop : ANR-09-JCJC-0048-01, 
coordonné par Xavier Caruso, projet de l'Agence Nationale de la Recherche.}}
 \author{C.~Huyghe and N.~Wach} 
\maketitle
\selectlanguage{francais}
\selectlanguage{english}
\begin{abstract} In 2003, Kedlaya gave an algorithm to compute the zeta function
associated to a hyperelliptic curve over a finite field, by computing the rigid cohomology of the curve. 
Edixhoven remarked that it is actually possible to compute the crystalline cohomology of
the curve, which is a lattice in the rigid cohomology. 
Following a method of Wach, we first explain how to use this lattice to compute the $(\varphi,\Gamma)$-module associated to an
hyperelliptic curve. We also explain an alternative way to get the $(\varphi,\Gamma)$-module mod $p$ that relies on
the Deligne-Illusie morphism.
\end{abstract}
\selectlanguage{francais}
\begin{abstract} En 2003, Kedlaya a donné un algorithme pour calculer la fonction zêta 
d'une courbe hyperelliptique sur un corps fini, en calculant la cohomologie rigide de la courbe. Edixhoven a
remarqué qu'on pouvait en fait calculer la cohomologie cristalline de la courbe, qui est
un réseau dans la cohomologie rigide. Dans un premier temps, nous expliquons
 comment utiliser ce réseau pour calculer le $(\varphi,\Gamma)$-module associé à une courbe hyperelliptique,
en suivant une méthode due à Wach. Ensuite nous donnons une autre méthode pour faire le calcul du 
$(\varphi,\Gamma)$-module mod $p$ reposant sur le morphisme de Deligne-Illusie.
\end{abstract}
\let\thefootnote\relax\footnote{MSC classification 2010 : 11S23, 14F30, 14F40}
%
%Ceci change la police.  \fontencoding{OT1}% \fontfamily{ptm}% \fontseries{m}% \selectfont

%
\tableofcontents
%}}}
\section*{Introduction}
%{{{

Soient $k$ un corps fini de caractéristique $p$ différente de $2$, de cardinal $p^n$, $W=W(k)$,
l'anneau des vecteurs de Witt de $k$,
 $K$ le corps des fractions de $W$, $S$ le schéma $\spec\,W$, 
$X$ une courbe hyperelliptique lisse sur $W(k)$, de genre $g$,
$X_k$, resp. $X_K$ la fibre spéciale, resp. la fibre générique de $X$. Le but de ce texte
est d'expliquer comment calculer modulo $p^i$ le $(\varphi,\Gamma)$-module associé à la 
courbe $X_K$, c'est-à-dire le  $(\varphi,\Gamma)$-module associé à la représentation
galoisienne
$H^1_{et}(X_{\ovK},\mathbf{Q}_p)$. Nous donnons aussi une autre façon de calculer 
modulo $p$ ce $(\varphi,\Gamma)$-module.

La méthode générale modulo $p^i$ exposée en ~\ref{methode_kedlaya} 
repose sur le calcul par Kedlaya \cite{Kedlaya_hyperell}, de la cohomologie rigide d'une courbe
hyperelliptique et de l'action du Frobenius sur cette cohomologie. Edixhoven 
a ensuite remarqué dans \cite{Edixhoven03pointcounting} que l'on pouvait calculer un réseau (en fait la
cohomologie cristalline) dans la cohomologie rigide des courbes hyperelliptiques, peut-être inspiré par un argument
de Katz dans le cas des courbes elliptiques qu'on trouve dans A1.2 de \cite{katz-pad_mod_scheme_form}. 
Bogaart a enfin donné la
démonstration du résultat d'Edixhoven dans un preprint non publié \cite{Bogaart}. Cette 
approche a été systématisée par Lauder \cite{Lauder-rec-zeta}, ainsi que par Abbott, Kedlaya et Roe
\cite{Abbott-Kedlaya-Roe-Pic-Numb}. Même si on peut calculer la cohomologie cristalline d'une courbe hyperelliptique,
on ne dispose pas cependant de calcul purement cristallin du Frobenius : pour ce
faire, il faut utiliser la méthode de Kedlaya. Le calcul du $(\varphi,\Gamma)$-module associé 
à la courbe $X_K$ fait intervenir un algorithme dû à Wach \cite{Wa97} et nécessite de
disposer d'un réseau stable par le Frobenius dans la cohomologie rigide de $X_k$, d'où
l'intérêt de la remarque d'Edixhoven. Le résultat que nous prouvons est le suivant: connaissant la matrice du Frobenius
dans une base adapt\'ee \`a la filtration de Hodge de la cohomologie cristalline, il est possible de d\'eterminer
algorithmiquement le $(\varphi,\Gamma)$-module associé à la courbe $X$. Si la courbe est de genre $g$, la
  complexit\'e de l'algorithme calculant le $(\varphi,\Gamma)$-module modulo $(p^i,T^j)$ est au plus $O(p^{1+\varepsilon}n^{1+\varepsilon}g^{2,81}i^{2+\varepsilon}j^{2+\varepsilon})$. Il faut y rajouter la   complexit\'e de l'algorithme de Kedlaya pour calculer la matrice de Frobenius modulo $p^{i+1}$ qui est $O(p^{1+\varepsilon}n^{1+\varep}g^{2+\varep}i^{2+\varepsilon})$ et celle d'un changement de base.

On peut se poser une question plus faible et calculer seulement le
$(\varphi,\Gamma)$-module modulo $p$ associé à la courbe hyperelliptique. On donne en~\ref{frob_divise} un calcul direct 
 pour d\'eterminer la matrice du Frobenius divis\'e sur la cohomologie cristalline modulo $p$. Ce
calcul repose sur le morphisme de Deligne-Illusie \cite{deligne_illusie_dec_dr} et est une conséquence du résultat de~\cite{huyghe-wach_interp-crist}. L'algorithme de Kedlaya est alors remplacé par un algorithme dont la   complexit\'e est $O(p^{1+\varepsilon}n^{1+\varep}g^{2+\varep})$, comparable \`a celle de l'algorithme de Kedlaya. Cette m\'ethode a cependant l'avantage de s'adapter facilement \`a d'autres familles de courbes ({\it loc. cit.}) et de proposer directement la matrice du Frobenius divis\'e dans une base adapt\'ee \`a la filtration.

\vspace{+2mm}
Voici le contenu de cet article.
Après une première partie de notations, nous donnons dans une deuxième partie une
stratégie générale pour calculer la cohomologie de de Rham d'une courbe projective lisse sur $W$. En effet, 
d'après le théorème de comparaison de Berthelot, cette cohomologie de de Rham est
canoniquement isomorphe à la cohomologie cristalline de la courbe $X_k$. En particulier, cela
nous donne un réseau stable par le Frobenius dans la cohomologie rigide de $X_k$.
Cette stratégie de calcul
reprend celle de Bogaart-Edixhoven et est exposée ici dans le cadre des catégories
dérivées.
L'argument est simple : si on twiste le complexe de de Rham de $X$ par une puissance
suffisante du faisceau inversible associé au diviseur à l'infini, on obtient un complexe acyclique pour le foncteur
sections globales. Et le quotient des deux complexes est un complexe à support au point
à l'infini. Du coup, tout est facilement calculable dans la catégorie dérivée des
complexes de $W$-modules, pourvu que l'on sache déterminer les 
sections globales du module des différentielles. C'est ce que nous faisons dans la
troisième partie. 

Dans la quatrième partie, nous détaillons comment trouver le $(\varphi,\Gamma)$-module
associé à une courbe. On consid\`ere le produit tensoriel sur $W$ de l'anneau $S=W[T]$ et du $W$-module
$H^1_{cris}(X)$, sur lequel on d\'efinit une action naturelle de $\varphi$; l'algorithme, d\'ej\`a expliqu\'e dans
\cite{Wa97}, consiste \`a d\'eterminer la matrice d'un g\'en\'erateur de $\Gamma$ par d\'evissages modulo les puissances
de $p$ et les puissances de $T$. Il suffit alors de calculer l'inverse de la matrice du Frobenius divis\'e modulo $p$. 

Nous appliquons ces résultats au cas des courbes hyperelliptiques dans la cinquième partie.
L'algorithme donnant le $(\varphi,\Gamma)$-module 
à une précision $(p^i,T^j)$ fixée reposant sur l'algorithme de Kedlaya est donné en ~\ref{methode_kedlaya}.
En \ref{frob_divise}, nous pr\'esentons une autre méthode pour calculer modulo $p$ la matrice du Frobenius divisé sur la cohomologie cristalline, c'est-\`a-dire la matrice qu'il faut inverser pour déterminer le $(\varphi,\Gamma)$-module modulo $p$.  

\vspace{+2mm}
Nous remercions le Referee pour ses remarques constructives sur divers aspects de ce texte, 
ainsi que Jérémy Leborgne, et Marcela Szopos qui nous ont aidées à éclaircir les problèmes de
complexité.
%}}} 

\section{Conventions, notations}\label{notations}
Si $X$ est un $S$-schéma, nous noterons $X_0$ la fibre spéciale et $X_K$ la fibre
générique.

Une courbe $X$ sur $S$ est par définition un $S$-schéma noetherien de dimension $1$ qui
est de type fini sur $S$, irréductible, dont les fibres sur $S$ (c'est-à-dire la fibre spéciale et la
fibre générique) sont réduites et connexes, géométriquement 
réduites et géométriquement connexes. On suppose de plus qu'il existe une 
section $i_s$ : $S\hrig X$ correspondant à un diviseur relatif de degré $1$ sur $S$.

Sous ces hypothèses (3.3.21 de \cite{Liu}), $H^0(X_0,\OO_{X_0})=k$ et
$H^0(X_K,\OO_{X_K})=K$.

On dit qu'une courbe $X$ est de genre $g$ 
si la fibre générique et la fibre spéciale sont de genre arithmétique égal à $g$.
Dans le cas où $X$ est lisse, ce qui sera toujours le cas ici, l'égalité de ces deux
nombres est automatique, par invariance de la caractéristique d'Euler Poincaré (5.3.28 de
\cite{Liu}), et platitude de $\OO_X$ sur $W$. De plus, si $X$ est lisse, $X$ est un schéma
intègre.

Dans cet article, la catégorie $D^b(W)$ désignera la catégorie dérivée des complexes de 
$W$-modules sur une courbe $X$, à cohomologie bornée. 

Nous disons qu'une courbe lisse $X$ vérifie la
condition de Mazur explicitées dans l'appendice de ~\cite{Mazur-hodge-filt},
si
\begin{gather}\label{conditions-mazur}
 \forall i\in\{0,1\},\forall j\in\{0,1\}\;H^i(X,\Omega^j_{X}) \text{ est un }W\text{
module libre }. 
\end{gather}
Si ces conditions sont vérifiées, alors, par définition, les modules $E_1^{j,i}$ de  la suite spectrale de Hodge vers de Rham
sont des $W$-modules libres et donc la suite spectrale de Hodge vers de Rham 
dégénère d'après loc. cit..

\section{Méthode de calcul du complexe de de Rham d'une courbe} \label{complexesDR}  
Dans cette section $X$ est une courbe lisse de genre $g$ comme définie en ~\ref{notations}. 
Si $C$ est le complexe de de Rham $$C \,\colon\, 0 \rig \OO_X \rig
\Omega^1_X \rig 0, $$ dont les termes sont en degré $0$ et $1$, la cohomologie de de Rham
de $X$ est la cohomologie du complexe $R\Gamma(X,C)$. Nous
considérons aussi le complexe de de Rham twisté $C(nD)$ où $D$ est un diviseur de $X$ $$C(nD)\,\colon\, 0 \rig \OO_X((n-1)D) \rig
\Omega^1_X(nD) \rig 0. $$

La filtration bête du complexe $C$ induit la suite spectrale de Hodge vers de Rham. La proposition suivante rappelle un critère de
dégénerescence standard.
\begin{surprop} \label{suite-spectrale}\be 
\item[(i)]Sous nos hypothèses, la condition de Mazur ~\ref{conditions-mazur} est vérifiée. En particulier, la
suite spectrale de Hodge vers de Rham dégénère au niveau $E_1$. De plus $H^0(X,\OO_X)=W$. 
\item[(ii)]Les modules
$H^0_{DR}(X)$ et $H^2_{DR}(X)$ s'identifient à $W$, le module $H^1_{DR}(X)$ est un $W$-module libre.
\ee 
\end{surprop} 
\begin{proof}Le terme
général de la suite spectrale de Hodge vers de Rham est $E_1^{j,i}=H^i(X,\Omega^j_X)$, les différentielles de la suite spectrale proviennent
des flèches canoniques $H^i(X,\Omega^j_X)\rig H^{i}(X,\Omega^{j+1}_X)$ et la dégenerescence de la suite spectrale équivaut, dans le cas d'une
courbe lisse $X$, à la nullité de la flèche $H^1(X,\OO_X)\rig H^1(X,\Omega^1_X)$.

Montrons que les modules $H^i(X,\Omega^j_X)$ sont des $W$-modules libres. Ceci
résulte par exemple du lemme 3.1 de ~\cite{Bogaart}, dont nous reproduisons la démonstration ici. D'après le théorème de semi-continuité
(5.3.20 de ~\cite{Liu}), il suffit de vérifier que les groupes de cohomologie $H^i(X_k,\Omega^j_{X_k})$ et $H^i(X_K,\Omega^j_{X_K})$ ont même
dimension pour $i,j\in\{0,1\}$. La
caractéristique d'Euler-Poincaré des faisceaux $\OO_{X_K}$
et $\OO_{X_0}$ est la même. De plus, $H^0(X_0,\OO_{X_0})=k$ et $H^0(X_K,\OO_{X_K})=K$ ont
même dimension sur $k$ et $K$ respectivement, de sorte que $H^1(X_0,\OO_{X_0})$ et 
$H^1(X_K,\OO_{X_K})$ ont aussi même dimension. Par dualité de Serre, les modules 
$H^i(X_0,\Omega^1_{X_0})$ et $H^i(X_K,\Omega^1_{X_K})$ ont aussi même dimension. Ainsi 
la courbe $X$ vérifie la condition de Mazur.

Considérons maintenant $M=H^0(X,\OO_X)$: c'est un $W$-module libre $M$ tel que 
$M\ot K$ est isomorphe à $H^0(X_K,\OO_{X_K})=K$, et qui est donc isomorphe à $W$. 
 Par dualité de Serre (III, \S 11 de \cite{hartshorne_res_duality}) 
et compte tenu du fait que $H^1(X,\Omega^1_X)$ est un $W$-module libre, nous trouvons que 
$H^1(X,\Omega^1_X)\simeq W$.
 Ceci mis ensemble montre (i). La dégénerescence de la suite spectrale de Hodge vers de
Rham implique que l'on a une suite exacte courte 
$$0\rig H^0(X,\Omega^1_X) \rig H^1_{DR}(X)\rig H^1(X,\OO_X) \rig 0,$$ de sorte que 
le $W$-module $H^1_{DR}(X)$ est libre.
De plus $H^0_{DR}(X)\simeq H^0(X,\OO_X)\simeq W$, et $H^2_{DR}(X)\simeq H^1(X,\Omega^1_X)=W$.  \end{proof}

Considérons maintenant un diviseur relatif $D$ sur $X$. Alors nous pouvons compléter la
proposition précédente par la
\begin{surprop}\label{acyc} \begin{enumerate} \item Si $D$ est de degré $>0$, alors
$H^0(X,\Omega^1_X(D))$ est
un $W$-module libre et $H^1(X,\Omega^1_X(D))=0$. 
\item Si $D$ est de degré $\geq 1$ et si $n$ est un entier tel que $(n-1)deg(D)>2g-2$, alors
$H^0(X,\OO_X((n-1)D))$ est un $W$-module libre et le complexe 
$C(nD)$ est à termes acycliques pour le foncteur $\Gamma$.
 \end{enumerate}
\end{surprop}
\begin{proof}
Là encore, nous suivons 3.1 de \cite{Bogaart}. Si $D$ est de degré $>0$, alors $dim\,
H^0(X_0,\OO_{X_0}(-D))=dim\, H^0(X_K,\OO_{X_K}(-D))=0$ et donc, par dualité de Serre, 
$dim\, H^1(X_0,\Omega^1_{X_0}(D))=dim\, H^1(X_K,\Omega^1_{X_K}(D))=0$, ce qui entraîne
comme précédemment que 
$H^1(X,\Omega^1_X(D))=0$ car ce module est un $W$-module de type fini.
 Par invariance de 
la caractéristique d'Euler-Poincaré, il y a donc aussi égalité des dimensions des
$H^0(X_0,\Omega^1_{X_0}(D))$ et de $H^0(X_K,\Omega^1_{X_K}(D))$, d'où le fait que 
$H^0(X,\Omega^1_{X}(D))$ est un $W$-module libre et (i). En
particulier,
avec les hypothèses de (ii), on sait que $H^1(X,\Omega^1_X(nD))=0$.
  Le raisonnement est identique pour (ii), puisque
$H^0(X_0,\Omega^1_{X_0}(-(n-1)D))$
(resp. $H^0(X_K,\Omega^1_{X_K}(-(n-1)D))$) est nul, car le degré de $\Omega^1(-(n-1)D)$ est $<0$.
Par dualité de Serre, $H^1(X_0,\OO_{X_0}((n-1)D))$ (resp. $H^1(X_K,\OO_{X_K}((n-1)D))$) est nul, et 
ceci montre que $H^1(X,\OO_X((n-1)D))=0$. Par invariance de la caractéristique
d'Euler-Poincaré, on obtient aussi que $H^0(X_0,\OO_{X_0}((n-1)D))$ et
$H^0(X_K,\OO_{X_K}((n-1)D))$ ont même dimension et donc que $H^0(X,\OO_X((n-1)D))$
est un $W$-module libre.
\end{proof}
\noindent{\bf Remarque.} On notera que la différentielle $$
d_{1}^{\{0,0\}}\,\colon\,H^0(X,\OO_X((n-1)D))\rig H^0(X,\Omega^1_X(nD)),$$ n'est
pas nulle en général (par exemple pour $\mathbf{P}^1_S$, $n=2$, $D$ le point à l'infini) et donc que la suite spectrale associée à la
filtration bête du complexe $C(nD)$ ne dégénère pas en général.

Fixons maintenant une section $i_s$ : $S\hrig X$ correspondant à un diviseur relatif de
degré $\geq 1$ et $n$ strictement supérieur à $2g-2$. La
suite exacte de complexes $$0 \rig C \rig C(nD) \rig C(nD)/C\rig 0,$$ donne lieu à un triangle distingué (T') dans $D^b(W)$.

Si $t$ est un paramètre local définissant la section $s$, le complété de $\OO_{X,t}$ pour la topologie $t$-adique s'identifie à l'anneau de
séries formelles $W[[t]]$ (chap. 8, \S 5 Anneaux locaux réguliers, thm 2 de \cite{bourbaki_alc_89}) et le complexe de faisceaux gratte-ciels
$C(nD)/C$ au complexe
\begin{eqnarray*} 
 \bigoplus_{i=1}^{n-1}W t^{-i} \sta{d}{\rig} \bigoplus_{i=1}^n W t^{-i} dt.
\end{eqnarray*}
 Comme $d(t^{-i})=(-i)t^{-(i+1)}$, le
complexe $R\Gamma(X,C(nD)/C)$ est concentré en degré $1$, où il s'identifie à 
\begin{eqnarray} \label{cokerd}
coker(d)=\bigoplus_{i=1}^n (W/(i-1)W) t^{-i} dt.
\end{eqnarray}
 En outre,
le complexe $R\Gamma(X,C(nD))$ est représenté par le complexe 
$$0 \rig \Ga(X,\OO_X((n-1)D)) \rig \Ga(X,\Omega^1_X(nD)) \rig 0. $$

Finalement, pour calculer la cohomologie de de Rham, on dispose d'un triangle distingué 
\begin{eqnarray}\label{triangle}
\quad R\Gamma(X,C)\rig  R\Gamma(X,C(nD))\rig
R\Gamma(X,C(nD)/C) \sta{+}{\rig} R\Gamma(X, C).
\end{eqnarray}
Le but de cet article est de décrire ce triangle dans le cas des courbes hyperelliptiques, pour $D$ correspondant au point à l'infini
et $n=2g$, où $g$ est le genre de la courbe.
Pour les calculs, nous utiliserons le lemme suivant.
\begin{surlem}\label{sections_saturees} Soient $U$ un ouvert affine de $X$, supposé
irréductible, tel que $U_0$
 soit un ouvert non vide de la fibre spéciale $X_0$, et $\EE$ un faisceau localement
 libre de rang fini sur $X$.
\be \item[(i)]
 Soit $f\in\EE(U)$
tel qu'il existe $l\in\Ne$ tel que $p^lf\in\EE(X)$, alors $f\in\EE(X)$.
\item[(ii)] De plus $\EE(X)\subset \EE(U)$ et le quotient est sans $p$-torsion.
\item[(iii)] Soit $M$ un réseau de $\EE(X)$, tel que $\EE(U)/M$ soit sans
$p$-torsion, alors $M=\EE(X)$.
\ee
\end{surlem}
\begin{proof} Notons $j$ l'inclusion de $U \hrig X$. Remarquons que, puisque $X$ et $X_0$
sont des schémas intègres, les hypothèses impliquent que $$\EE/p\EE \hrig 
j_*\EE_{|U}/p\EE_{|U}.$$ En effet, $X_0$ est recouvert par des ouverts affines sur 
lesquels $\EE/p\EE$ est un $\OO_{X_0}$-module libre, si bien que l'assertion se ramène 
à la même assertion pour $\EE=\OO_{X}$, qui résulte de l'intégrité de $X_0$ et
du fait que $U_0$ est dense dans $X_0$. En passant aux sections globales, et en tenant compte du
fait que l'ouvert $U$ est affine, on a une inclusion  $$r \;\colon \; \EE(X)/p\EE(X) \hrig
\EE(U)/p\EE(U).$$
Soit $f\in\EE(U)$ comme dans l'énoncé, tel que $p^lf\in \EE(X) $ avec $l\geq 1$. 
Nous procédons par récurrence sur $l$, le cas $l=0$ étant évident. Posons
$f'=p^{l-1}f\in\EE(U)$. Alors $h=pf'\in \EE(X)$, et $r(h)=0$, ce qui montre qu'il existe 
$h'\in\EE(X)$ tel que $h=pf'=ph'$, or $\EE(X)$ est sans $p$-torsion car $\EE$ est
localement libre, de sorte que ceci implique que $f'=h'\in\EE(X)$. Ainsi 
$p^{l-1}f\in\EE(X)$ et l'hypothèse de récurrence entraîne que $f\in\EE(X)$.
Le (ii) résulte du fait qu'on a une inclusion $$\EE \hrig 
j_*\EE_{|U},$$ d'où l'inclusion des sections globales. Le fait que le quotient soit sans
$p$-torsion est une simple traduction de (i). Le (iii) vient du fait que l'on a une 
injection $\EE(X)/M \hrig \EE(U)/M$, de sorte que $\EE(X)/M$ est sans torsion et comme 
$M$ est un réseau de $\EE(X)$, on a l'égalité $M=\EE(X)$.
\end{proof}
Nous aurons aussi besoin de quelques considérations de base à propos de l'action de
l'involution hyperelliptique. Soient $V$ une $W$-algèbre, $M$ un $V$-module muni d'une
action $V$-linéaire 
d'une involution $u$ (telle que $u^2=Id$). On pose alors :
$$M^{-}=\{x\in M\,|\,u(x)=-x\}, \quad M^{+}=\{x\in M\,|\,u(x)=x\}.$$
Comme $car(k)\neq 2$, on a un isomorphisme canonique
$$\xymatrix{M \ar@{->}[r]^(.4){\sim}& M^{+}\bigoplus M^{-}\\
          x  \ar@{|->}[r]& (x+u(x),x-u(x)),  }$$
et le foncteur $M\mapsto M^{-}$ est exact, de même que $M\mapsto M^+$. 
Nous en déduisons le fait suivant.
Soit $V\rig V'$ un morphisme de $W$-algèbres. Si $M$ est muni d'une
involution $V$-linéaire, $u$, alors $1\ot u$ est une involution $V'$-linéaire du
$V'$-module $M\ot_V V'$, et
 le morphisme canonique $M^{-}\ot_V V' \rig (M\ot_V V')^{-}$ est un
isomorphisme (resp. avec $M^{+}$).

\section{Cohomologie de de Rham des courbes hyperelliptiques} 
\label{coh_DR_courbes-hyperell}
\subsection{Equations} \label{equations-diff2}
\subsubsection{Données}
Soit $P$ un polynôme unitaire de $W[X]$ de degré impair $d=2g+1$, tel que $P$ est séparable comme polynôme à
coefficients dans $K$ et tel que la classe de $P$ dans $k[X]$ est aussi séparable. On note $P(X)=X^d+\sum_{l=1}^{d-1}a_l X^l$.

Définissons, suivant 7.4.24 de \cite{Liu}, la courbe hyperelliptique $X$ suivante. Soient $[u,v]$ les coordonnées projectives sur
$\mathbf{P}^1_W$, $x=v/u$ une coordonnée fixée sur $D_+(u)$, $x_1=u/v$ une autre fixée sur $D_+(v)$. La courbe $X$ est un revêtement $h$ de
degré $2$ de $\mathbf{P}^1_W$, réunion de deux ouverts affines $U=h^{-1}(D_+(u))$ et $W=h^{-1}(D_+(v))$ donnés par $$ U={\rm
Spec}W[x,y]/y^2-P(x),$$ $$   V={\rm Spec}W[x_1,t]/t^2-P_1(x_1),$$ où l'on a posé $P_1(x_1)=x_1^{d+1}P(1/x_1)=x_1 R_1(x_1)$, $t=y/x^{g+1}$.
 
Les conditions sur $P$ assurent que la courbe $X$ ainsi définie est lisse et irréductible.
De même, les fibres génériques et spéciales sont lisses, connexes et géométriquement
connexes. Nous pouvons donc appliquer les résultats de ~\ref{complexesDR}. On note $X_0$ la fibre spéciale de $X$ et $X_K$ la
fibre générique de $X$. 

On vérifie facilement que comme dans le cas d'un corps le lieu de ramification de $h$ est $ V(P)\bigcup V(x_1)$. 

Dans la suite nous noterons $A$ (resp. $A'$)la W-algèbre $W[x,y]/y^2-P(x)$, qui est un $W[x]$-module libre de rang $2$ de base $1$, $y$.

L'application définie sur $U$ par $\iota(x,y)=(x,-y)$ et sur $V$ par $\iota(x_1,t)=(x_1,-t)$ définit une involution de $X$ dont les points
fixes sont le lieu de ramification de $h$.
\subsubsection{Situation en l'infini} Par convention, le point à l'infini noté $\inft$ de $\mathbf{P}^1_W$ correspond à $x_1=0$, et le point à
l'infini aussi noté $\inft$ de $X$ est le point de $V$ vérifiant $x_1=0$.  Remarquons que $R_1(x_1)=1+\sum_{l=1}^d a_{d-l}x_1^l$ est de
valuation $0$ en $x_1$, en d'autres termes $P'_1$ est inversible de l'anneau local $\OO_{X,\inft}$. Un générateur du faisceau localement libre
$\Omega^1_V$ est $dx_1/t$ d'après 6.4.14 de \cite{Liu}.  Dans le module
$\Omega^1_{V,\inft}$ on a la relation $dx_1/2t=(P'(x_1)^{-1})dt$ de
sorte que $t$ est aussi un paramètre à l'infini sur $X$.
Donnons quelques formules
\begin{equation}\label{form1}
\begin{split} 
\frac{dx}{2y}=-\frac{x_1^{g-1}dx_1}{2t}\,\in \Omega^1_V,\\
x =x_1^{-1} = \frac{1}{t^2}R_1(x_1),\\
R_1(x_1) = x_1^{2g+1}P(x).
\end{split}
\end{equation}

 Commençons par calculer les sections globales de
$\Omega^1_X$, $\Omega^1_X(2g\inft)$ et préciser l'action de l'involution
 hyperelliptique $\iota$ sur ces modules. 
\subsection{Calculs de sections globales}
\begin{prop} \label{sect_glob_omega}
Le $W$-module $H^0(X,\Omega^1_X)$
 est libre de rang $g$ de base
$$\omega_i=\frac{x^idx}{y}=-\frac{x_1^{g-1-i}dx_1}{t}, \, 0\leq i \leq g-1.$$
\end{prop} 
%}}
Sur un corps, cela résulte de 7.4.26 de \cite{Liu}. D'après 6.4.14 de
\cite{Liu}, le faisceau $\Omega^1_U$ (resp. $\Omega^1_V$) est libre, engendré par $\omega_0$ 
(resp.  $\omega_{g-1}$) de sorte que les éléments
considérés sont clairement des éléments linéairement indépendants de $H^0(X,\Omega^1_X)$.
Soit maintenant $M$ le sous-module 
des sections globales engendré par les $\omega_i$
pour $0\leq i \leq g-1$, qui est a priori un réseau de 
$H^0(X,\Omega^1_X)$. Alors $\Omega^1_X(U)/M$ est libre sur $W$, de base 
les $x^j \omega_0$ pour $j\geq g$ et les  $x^jy \omega_0$ pour $j\geq 0$, si bien que 
$M=H^0(X,\Omega^1_X)$ d'après le lemme ~\ref{sections_saturees}.

En particulier, $\iota$ agit par $-Id$ sur les sections globales de $\Omega^1_X$. On en
déduit le
\begin{cor} \label{actionDR}
L'involution $\iota$ agit par $-Id$ sur $H^1_{DR}(X)$.  \label{subsubsection-iota} \end{cor} 
\begin{proof} D'après la dégénérescence de la suite spectrale de Hodge vers de Rham, 
on a une suite exacte courte $0\rig H^0(X,\Omega^1_X) \rig
H^1_{DR}(X)\rig  H^1(X,\OO_X)\rig 0.$ Nous venons de voir que $\iota$ agit par 
$-Id$ sur $H^0(X,\Omega^1_X)$. Le module $ H^1(X,\OO_X)$
est le dual de Poincaré de $H^0(X,\Omega^1_X)$ et donc $\iota$ agit aussi par $-Id$ sur ce module. 
Il s'ensuit que $\iota$ agit par $-Id$ sur $H^1_{DR}(X)$.
\end{proof}
\begin{prop}  \label{deg_suite_spec_tw}
 \be
\item[(i)] Soit $n\in\Ne$ tel que $n\geq 2g-1$, 
le $W$-module $H^0(X,\OO_X(n\inft))$ est libre de base les $x^i$ pour
$i\in\{0,\ldots,[n/2]\}$ et les $yx^i$ pour $i\in\{0,\ldots,[(n-1)/2]-g\}$.
\item[(ii)] En particulier, le $W$-module $H^0(X,\OO_X((2g-1)\inft))$ 
est libre de base les $x^i$ pour
$i\in\{0,\ldots,g-1\}$ et $$ H^0(X,\OO_X((2g-1)\inft)^{-}=0.$$
\ee
\end{prop}
\begin{proof} Comme le faisceau $\OO_X(n\inft)$ est localement libre, le module 
$H^0(X,\OO_X((2g-1)\inft))$ est un $W$-module libre. Pour $n\geq 2g-1$, la formule de
Riemann Roch donne qu'il est de rang $n+1-g=[n/2]+1+[(n-1)/2]-g+1$, qui est le nombre
d'éléments considérés. Ces éléments sont bien des sections globales de
$H^0(X,\OO_X(n\inft))$ 
puisque, en vertu de ~\ref{form1}, 
$$\begin{cases} \mbox{si } 0\leq i \leq \left[\frac{n}{2}\right],\,
x^i=\frac{R_1(x_1)^i}{t^{2i}}\\
\mbox{si } 0 \leq i \leq \left[\frac{n}{2}\right] -g-1,\,
x^iy=\frac{R_1(x_1)^{g+1+i}}{t^{2(i+g)+1}}.
\end{cases}$$
Notons alors $M$ le sous $W$-module de $H^0(X,\OO_X(n\inft))$ engendré par ces sections.
Comme $\OO_X(U)/M$ est un $W$-module libre, on voit grâce à ~\ref{sections_saturees}, que 
$M=H^0(X,\OO_X(n\inft))$. L'assertion (ii) n'est qu'un cas particulier de (i).
\end{proof}
Nous en déduisons le corollaire.
\begin{cor} $$\forall 0 \leq i \leq g-2,\, x^idx\in\Ga(X,\Omega^1_X(2g\infty))^+.$$
\end{cor} 
%}}
\begin{proof} En dérivant les sections globales de $\OO_X((2g-1)\inft)$, on trouve des
sections globales de $\Omega^1_X(2g\infty)$, lesquelles sont invariantes sous l'action de
l'involution $\iota$. De plus, en appliquant le lemme~\ref{sections_saturees}, on voit que si $ix^{i-1}dx\in
\Ga(X,\Omega^1_X(2g\infty))$, alors $x^{i-1}dx \in\Ga(X,\Omega^1_X(2g\infty))$.
\end{proof}

\noindent{Remarque.} Au passage, on constate que si la suite spectrale associée au
complexe filtré $C(2g\inft)$ ne dégénère pas, la suite spectrale obtenue à partir de 
celle-ci après application de $\iota=-Id$ dégénère car $car(k)\neq 2$. L'aboutissement de cette suite
spectrale est la partie de $R\Gamma(X,C(2g\inft))$ sur laquelle $\iota$ agit par $-Id$.

\begin{prop}\label{desc_base_inf} Le $W$-module
$H^0(X,\Omega^1_X(2g\infty))^{-}$ est libre de rang $2g$ de base
$$\omega_i=\frac{x^idx}{y}=-x_1^{2g-1-i}R_1(x_1)^g\frac{1}{t^{2g}}\frac{dx_1}{t}, \, 0\leq i \leq 2g-1.$$
\end{prop}
\begin{proof} Les éléments considérés sont clairement des sections globales de $\Omega^1_X(2g\infty))$, qui est un $W$-module sans torsion,
puisque le faisceau $\Omega^1_X(2g\infty))$ est localement libre de rang $1$. Calculons le rang de ce module. En passant à la suite exacte
longue de cohomologie de la suite exacte courte $$0\rig \Omega^1_X \rig \Omega^1_X(2g\infty) \rig
t^{-2g}\OO_{X,\inft}dt\left/\OO_{X,\inft}dt\right. \rig 0 ,$$ et en appliquant le foncteur exact $\iota=-Id$, on trouve la suite exacte
\begin{gather}\label{appli_v}
0\rig H^0(X,\Omega^1_X) \rig H^0(X,\Omega^1_X(2g\infty))^{-}\sta{v}{\rig} \left(\bigoplus_{j=1}^{2g}W\,t^{-j} dt\right)^{-}\rig
H^1(X,\Omega^1_X)^{-}.
\end{gather}
 Le terme $H^1(X,\Omega^1_X)$ est dual de $H^0(X,\OO_X)=W$ et donc $\iota$ agit trivialement sur ce module si bien que
le dernier terme de la suite est nul. Le troisième terme de cette suite exacte est $\bigoplus_{i=1}^{g}W.t^{-2i}dt$ puisque
$\iota(dt)=-dt$. Le rang cherché est donc $2g$. 

Notons $M$ le sous-module engendré par les éléments $\omega_i$, ce module est de rang $2g$ car les $\omega_i$ sont linéairement indépendants
sur $W$ dans $\Omega^1_X(U)$.  Le même raisonnement que précédemment montre que $M=H^0(X,\Omega^1_X(2g\infty))$ et finalement que le rang de
$H^0(X,\Omega^1_X(2g\infty))$ est $2g$.  \end{proof}

 Remarquons que $v(\omega_i)=0$ si $ 0\leq i \leq g-1$.

Dans tous les cas portant sur la caractéristique de $k$, le module 
$M=H^0(X,\Omega^1_{X}(2g\infty))$ est libre de rang $3g-1$ par le théorème de 
Riemann-Roch et l'énoncé qui suit se démontre exactement de la même façon
que le précédent.
\begin{prop} \label{desc2_base_inf_d2}Si $g\geq 2$ 
$$H^0(X,\Omega^1_{X}(2g\infty))=\bigoplus_{i=0}^{2g-1}W\cdot\frac{x^i dx}{y}
\bigoplus \bigoplus_{i=0}^{g-2}W\cdot x^i dx.$$
\end{prop} 
%}}
Pour une courbe elliptique, seul le premier terme apparaît. On remarquera que les éléments 
$x^i dx$ sont invariants par l'involution hyperelliptique $\iota$.
\subsection{Calcul explicite de la cohomologie de de Rham des courbes hyperelliptiques}
\subsubsection{Calcul du passage au complété de l'anneau local en l'infini}
\label{complete_infini}
Décrivons d'abord explicitement l'homomorphisme d'anneaux locaux
$$\OO_{X_\inft}\rig W[[t]].$$ 
 Notons $C=\OO_{X_\inft}$ l'anneau local en $\infty$ de la courbe $X$, de sorte 
que  
$$ C=\left[ W[x_1,t]\left/ t^2-x_1R_1(x_1)\right.\right]_{(pW[x_1,t]+tW[x_1,t])},$$
c'est-à-dire que $C$ est obtenu en localisant l'algèbre $W[x_1,t]\left/ t^2-x_1R_1(x_1)\right.$ par l'idéal 
maximal engendré par $t$ et $p$.
Le complété de $\what{C}$ pour la topologie $t$-adique est isomorphe à $W[[t]]$ via une 
application $\mu$ : $W[[t]]\rig \what{C}$. On note 
$\lam$ l'application $C\rig W[[t]]$ qui s'obtient comme $\mu^{-1}\circ Can$ où $Can$ est 
l'application canonique $C \rig \what{C}$, et par abus de notation on 
continue de noter $x_1$ l'élément $\lam(x_1)$. Cherchons $x_1$ sous la forme $x_1(t)=t^2U(t)$ avec 
$U(t)=\sum_i q_i t^i\in W[[t]]$. L'élément $U$ doit vérifier l'équation
\begin{equation*}(F)\quad t^2=(t^2U(t))\left[R_1(t^2U(t))\right],\end{equation*} c'est-à-dire
$1=U(t)R_1(t^2U(t))$. Notons $f(X)=X(R_1(t^2X))-1\in W[[t]][X]$, 
l'image de $f$ dans le corps résiduel $k$ de $W[[t]]$ est $R_1(0)X-1=X-1$ dont $1$ est une 
racine simple de sorte qu'on 
peut appliquer le lemme de Hensel et voir qu'il existe un unique 
$U(t)$ de W[[t]] (inversible) satisfaisant $f(U(t))=0$, et donc aussi un unique 
$x_1(t)=t^2U(t)\in W[[t]]$ vérifiant l'équation $(F)$. 

Remarquons de plus que $P'_1(x_1)\in 1+tW[[t]]$ est inversible dans 
$W[[t]]$, dont l'inverse sera noté $V(t)$. On a alors
$$\frac{dx_1}{2t}=\frac{dt}{P'_1(x_1)},$$
soit encore
$$\frac{dx_1}{2t}=V(t) dt.$$

Avec ces notations, du fait que $R_1(x_1)=U(t)^{-1}$ 
on voit que dans ${\what{\Omega}}^1_{X,\inft}=W[[t]]\ot 
{\what{\Omega}}^1_{X,\inft}$, on a l'égalité, en reprenant les notations 
de ~\ref{desc_base_inf}
$$\omega_i=-2t^{2(g-1-i)}U(t)^{g-1-i}V(t)dt.$$
Notons
\begin{eqnarray}\label{notation-u}\forall \, 0\leq i \leq g-1,\, 
2U(t)^{-1-i}V(t)&=&2+\sum_{l}u_{i,l}t^l,\, u_{i,l}\in W.
\end{eqnarray}
Alors, nous calculons $v(\omega_{g+i})$ où $v$ est l'application définie en ~\ref{appli_v}:
\begin{eqnarray*}\quad\forall \, 0\leq i \leq g-1,\,
v(\omega_{g+i})&=&-t^{-2(1+i)}\left(2+\sum_{l=1}^{2i+1} u_{i,2l}t^{2l} \right)dt,\\
 v(\omega_{g+i}) &=&-2t^{-2(i+1)} dt -\sum_{l=1}^i u_{i,2l}t^{2(l-i-1)}dt,\\
                 &=&-2t^{-2(i+1)} dt-\sum_{l=i}^1 u_{i,2(i+1-l)}t^{-2l} dt
\end{eqnarray*}

puisque les $v(\omega_{g+i})$ sont anti-invariants pour $\iota$ et que 
$\iota(t)=-t$. Pour terminer nous avons besoin de la proposition suivante

\begin{souscor}(Bogaart-Edixhoven, \cite{Bogaart},\cite{Edixhoven03pointcounting})\label{suite_cark_diff2} On a une suite exacte courte 
$$ (E)\quad 0\rig H^1_{DR}(X)\rig H^0(X,\Omega^1_X(2g\inft))^-\sta{v'}{\rig} \bigoplus_{j=1}^g 
\left(W/(2j-1)W\right)\,t^{-2j}dt\rig 0.$$
\end{souscor}
Reprenons le triangle ~\ref{triangle} de ~\ref{complexesDR} avec $D$ correspondant à la section à l'infini, et 
$n=2g-1$. On trouve le complexe exact
\begin{multline}\label{complexe1}
0 \rig H^0_{DR}(X) \rig R^0\Gamma(X,C(2g\inft))\rig 0 \rig H^1_{DR}(X) \rig
R^1\Gamma(X,C(2g\inft))  \\ \rig coker(d) \rig H^2_{DR}(X)\rig 0.
\end{multline}
Appliquons maintenant $\iota=-Id$ à ce complexe. Par dégénerescence de la suite spectrale
de Hodge vers de Rham, $ H^2_{DR}(X)$ s'identifie à $H^1(X,\Omega^1_X)=W$ sur lequel $\iota$ agit trivialement
et donc $ H^2_{DR}(X)^-$ est nul, tout comme $H^0_{DR}(X)^{-}$. 

Le complexe $R\Gamma(X,C(2g\inft))$ 
est le complexe (~\ref{acyc})
$$H^0(X,\OO_X(2g-1)\inft)\sta{d}{\rig} H^0(X,\Omega^1_X(2g\infty)),$$ de sorte que 
$R^1\Gamma(X,C(2g\inft))^-$ se réduit à $H^0(X,\Omega^1_X(2g\infty))^-$ d'après ~\ref{deg_suite_spec_tw}. 

Il reste à calculer $coker(d)^-$. Le conoyau 
$coker(d)$ est $\left(\bigoplus_{j=1}^{2g} W/(j-1)W\,t^{-j}dt \right)$. 

En passant 
à la partie $-$ de ce module, seuls subsistent les termes correspondant à $j$ pair, d'où le (i) de l'énoncé.

Reprenons les notations de~\ref{notation-u}.

Dans les bases $(\omega_0,\ldots,\omega_{2g-1})$ et $t^{-2l}dt$ pour $1\leq l\leq g$, 
la matrice de l'application $v$ (\ref{appli_v}) 
qui est dans $M_{g,2g}(W)$ la suivante. Notons 
$V=(v_{2l,i})$ avec $v_{2l,i}\in W$, avec $i\in \{0,\ldots ,2g-1\}$, 
$l\in\{1,\ldots, g\}$ cette matrice. Les coefficients sont 
\begin{eqnarray*}\forall i \leq g-1,\,\forall l,\, v_{2l,i}&=&0 \\
                  \forall 0\,\leq i \leq g-1,\, v_{2(i+1),g+i}&=&-2 \\
               \forall 0\,\leq i \leq g-1,\,\forall l\leq i,\, v_{2l,g+i}&=&-u_{i,2(i+1-l)}\\
                \forall 0\,\leq i \leq g-1,\,\forall l \geq i+1,\, 
v_{2l,g+i}&=& 0.
\end{eqnarray*}
La matrice $V$ est donc 
$$V=\left(\vcenter{
\xymatrix{ 0\ar@{.}[rrrr] &&&& 0 &
   -2&-u_{1,2}  &-u_{2,4}\ar@{.}[rr] && -u_{g-1,4g-1}\\
                   0\ar@{.}[rrrr]\ar@{.}[ddd] & &&& 0 \ar@{.}[ddd]&
    0& -2& -u_{2,2}\ar@{.}[rr]&& -u_{g-1,4g-3} \ar@{.}[dd]\\
                   &  &&  & 
    &0 \ar@{.}[dd] & 0   \ar@{.}[rrdd]\ar@{.}[dd]   & -2\ar@{.}[rrdd] &&   \ar@{.}[d]   \\
    &&&&&    & &&& -u_{g-1,2} \\
    0\ar@{.}[rrrr] & &&& 0 &0 
     & 0\ar@{.}[rr]&  &0& -2}}
\right)$$
et l'application $v'$ s'obtient à partir de $v$ par passage au quotient via la surjection
canonique 
$$\bigoplus_{i=0}^{g-1}W t^{-2(i+1)} dt \twoheadrightarrow
  \bigoplus_{i=0}^{g-1}W/(2i+1)W t^{-2(i+1)} dt.$$ 

On voit donc que $v$ induit une bijection notée $\ov$ 
$$ \bigoplus_{l=1}^g Wt^{-2l}\simeq \bigoplus_{i=0}^{g-1}W t^{-2(i+1)} dt ,$$
dont la matrice sera notée $\ovV$, et qui est triangulaire supérieure dans les bases
choisies. Un élément $\omega \in \bigoplus_{i=0}^{g-1}W t^{-2(i+1)} dt$ est dans $Ker(v')$ 
si et seulement si $\omega\in \ov^{-1}(Ker(s))$. Or on a l'identification 
$$Ker(s) =\bigoplus_{i=0}^{g-1}(2i+1)W t^{-2(i+1)} dt.$$
Pour $i\in\{0,\ldots,g-1\}$, définissons 
$$\omega'_{g+i}=\ovV^{-1}\left(\vcenter{
\xymatrix{
0\ar@{.}[d]  \\  0 \\ (2i+1)t^{-2(i+1)}dt\\
0\ar@{.}[d] \\ 0 }} 
\right)\in (2i+1)\omega_{g+i}\bigoplus_{l\leq i}W
\omega_{g+l}.$$
Nous avons alors la
\begin{prop}\label{base} Le $W$-module $H^1_{DR}(X)$ est libre de rang $2g$ de base des éléments
            \begin{eqnarray*} \frac{x^idx}{y}\quad {pour}\, 0\leq i \leq g-1, \\
                           \omega'_{g+i} \quad {pour} \, 0\leq i \leq g-1.
            \end{eqnarray*}
\end{prop} 
Si $p>2g-1$ le terme de droite de la suite exacte~\ref{suite_cark_diff2} est nul, ce qui donne 
l'énoncé suivant.
\begin{cor} Si $p>2g-1$ on a un isomorphisme
$$ H^1_{DR}(X)\simeq H^0(X,\Omega^1_{X}(2\infty))^-.$$
En particulier, si $X$ est une courbe elliptique, on a toujours cet isomorphisme.
\end{cor} 
%}}

\section{Action du Frobenius et $(\varphi,\Gamma)$-module associ\'e}
\label{phi_Gamma_courbes}
\subsection{Action du Frobenius}

\subsubsection{G\'en\'eralit\'es}
\label{psi-frob-divise}
Rappelons qu'un $\varphi$-module filtr\'e sur $W$ est la donn\'ee d'un $W$-module muni d'un endomorphisme $\varphi$ semi-lin\'eaire par rapport au Frobenius de $W$ et d'une filtration par des sous-$W$-modules, d\'ecroisssante, exhaustive et s\'epar\'ee.

La cohomologie de De Rham $H^1_{DR}(X)$ d'une courbe $X$ projective et lisse de genre $g$ est un $W$-module libre de rang $2g$, muni de la filtration de Hodge:
$$\begin{array}{lll}\fil^rH^1_{DR}(X) &=H^1_{DR}(X)\quad&\hbox{ pour }r\leq0,\\
&=H^0(X,\Omega^1_X)\quad&\hbox{ pour }r=1,\\
&=0 \quad&\hbox{ pour }r\geq2.\\
\end{array}$$
De plus, la cohomologie de De Rham de $X$ s'identifie \`a la cohomologie cristalline $H^1_{cris}(X_k)$ de $X_k$ \cite{Berthelot-coh_cris}; elle est donc munie d'un endomorphisme de Frobenius, induit par l'action du Frobenius sur $X_k$. On obtient ainsi un $\varphi$-module filtr\'e.

Lorsque les $W$-modules $H^p(X,\Omega^q_X)$ sont des $W$-modules libres et que la suite spectrale associ\'ee d\'eg\'en\`ere (Proposition \ref{suite-spectrale} pour une courbe hyperelliptique), on sait, par une application d'un th\'eor\`eme de Mazur \cite{Mazur-hodge-filt} et \cite{Fo83}, que le $\varphi$-module filtr\'e $H^1_{DR}(X)$ est un r\'eseau fortement divisible de $H^1_{DR}(X_K)$: pour tout entier $r$, chaque $\fil^rH^1_{DR}(X)$ est un facteur direct de $H^1_{DR}(X)$, on a 
$$\varphi(\fil^rH^1_{DR}(X))\in p^rH^1_{DR}(X)$$  et $$\sum_{r\in\bf Z}{1\over p^r}\varphi(\fil^rH^1_{DR}(X))=H^1_{DR}(X).$$ 

Consid\'erons $(e_i)_{1\leq i\leq 2g}$ une base de $H^1_{DR}(X)$ adapt\'ee \`a la filtration \cite{Wa97}; soit $r_i$ le plus grand entier tel que $e_i\in\fil^{r_i}H^1_{DR}(X)$ et pour tout entier $r\geq0$, posons $\varphi_r={1\over p^r}\varphi_{|_{\fil^r}}$. On a ainsi, pour $1\leq i\leq 2g$, si
$$e_i\in\fil^1H^1_{DR}(X), \quad r_i=1, \quad\hbox{ et }\quad\varphi_1(e_i)={1\over p}\varphi(e_i)$$ 
et, sinon, 
$$r_i=0\quad\hbox{ et }\quad\varphi_0(e_i)=\varphi(e_i).$$ 
Ecrivons, pour $1\leq j\leq 2g$, 
$$\varphi_{r_j}(e_j)=\sum_{i=1}^{2g}a_{ij}e_i$$
l'expression de $\varphi_{r_j}(e_j)$ dans la base $(e_i)$. Le $\varphi$-module filtr\'e $H^1_{DR}(X)$ \'etant fortement divisible, les applications $\varphi_r$ sont bien d\'efinies, 
%induisent une application $\psi : H^1_{DR}(X)\rig H^1_{DR}(X)$, 
et la matrice $A=(a_{ij})$ est une matrice \`a coefficients dans $W$, qui est inversible (\cite{Wa97}, 2.2.2).

\subsection{($\varphi,\Gamma$)-module associ\'e}
\label{phi-gamma-module}

Consid\'erons la ${\bf Z}_p$-extension cyclotomique $K_\infty$ de $K$ contenue dans $\bar K$, notons $\Gamma$ le groupe de Galois de $K_\infty/K$, qui est isomorphe \`a ${\bf Z}_p$, et choisissons-en un g\'en\'erateur topologique, not\'e $\gamma$. On munit l'anneau $S=W[[T]]$  d'une action de $\Gamma$  telle que $\gamma(T)-T=\alpha(p+T)T$, o\`u $\alpha$ est une unit\'e de $S$, qui est congrue \`a une unit\'e de ${\bf Z}_p$ modulo $T$ et \`a $1$ modulo $(p,T)$. Dans ces conditions, il existe une unique action de Frobenius, not\'ee $\varphi$, commutant \`a l'action de $\Gamma$, compatible avec le Frobenius sur $W$, qui rel\`eve l'\'el\'evation \`a la puissance $p$ modulo $p$ et telle que $\varphi(T)=u(p+T)^{p-1}T$, o\`u $u\equiv1$ modulo $pS$ (pour un expos\'e d\'etaill\'e des propri\'et\'es de S, voir \cite{Fo90}, 3.2, \cite{Fo94},  ou \cite{Wa97}, 3.1.1.).

Pour une courbe projective et lisse $X$ v\'erifiant les hypoth\`eses du th\'eor\`eme de Mazur \cite{Mazur-hodge-filt}, on d\'efinit sur le module $N=S\otimes_WH^1_{DR}(X)$ une action de $\varphi$ par
$$\varphi(1\otimes e_j)=(p+T)^{r_j}\sum_{i=1}^{2g}a_{ij}(1\otimes e_i),$$
l'action de $\Gamma$ \'etant induite par la proposition suivante (cf \cite{Wa97}, 3.1.4, th\'eor\`eme 3):

\begin{prop} \label{unicite-gamma} Il existe sur $N$ une unique action de $\Gamma$ commutant \`a $\varphi$ et triviale modulo $T$. \end{prop}

\begin{proof}
Nous reproduisons ici la d\'emonstration de {\it op.cit}, adapt\'ee au cas particulier d'une courbe.

Il suffit de conna\^\i tre la matrice du g\'en\'erateur $\gamma$ dans la base $(1\otimes e_i)_{1\leq i\leq 2g}$, not\'ee simplement $(e_i)_{1\leq i\leq 2g}$; la d\'emonstration se fait par d\'evissages modulo $T$ et modulo $p$ en d\'eformant la matrice de l'identit\'e $I_{2g}$, qui induit sur $N$ une action de $\gamma$ triviale modulo $T$, mais ne commutant pas \`a $\varphi$ a priori. Notons $\rho$ cette action:

$$\begin{array}{ll}\rho(\varphi(e_j))-\varphi(\rho(e_j))&=(\gamma(p+T)^{r_j}-(p+T)^{r_j})\sum_{i=1}^{2g}a_{ij}e_i\\
&=(p+T)^{r_j}\alpha_jT\sum_{i=1}^{2g}a_{ij}e_i\\
\end{array}$$
si $\alpha_j$ est l'unique \'el\'ement de $S$ tel que 
$$\gamma(p+T)^{r_j}=(p+T)^{r_j}(1+T\alpha_j)\quad.$$
 
Comme $S$ est complet pour la topologie $T$-adique, il suffit de v\'erifier 
le lemme suivant:

\begin{lem}
 Soient $n\in\bf N$ et $\rho$ un endomorphisme de $N$, 
semi-lin\'eaire par rapport \`a l'action de $\gamma$ sur $S$ tel que
$\rho(e_j)\equiv e_j\ \mod T$ et $$\rho(\varphi(e_j))\equiv\varphi(\rho(e_j))\ \mod T^n
(p+T)^{r_j}N$$ pour tout $j$; alors, il existe un endomorphisme $\rho^\prime$ de $N$, uniquement d\'etermin\'e
modulo $T^{n+1}$, semi-lin\'eaire
par rapport \`a l'action de $\gamma$ sur $S$ tel que, pour tout $j$: 
$$\begin{array}{ll} &\rho^\prime(e_j)\equiv\rho(e_j)\ \mod T^nN\\
\hbox{et} &\rho^\prime(\varphi(e_j))\equiv\varphi(\rho^\prime(e_j))\ \mod T^{n+1}(p+T)^{r_j}N\quad.\\\end{array}$$
\end{lem}

D\'emontrons la proposition. On pose, pour tout $j$, $1\leq j\leq 2g$: 
$$\rho(\varphi(e_j))-\varphi(\rho(e_j))=T^n(p+T)^{r_j}b_j$$
o\`u $b_j$ est un \'el\'ement de $N$ et on cherche la matrice $G'$ de $\rho^\prime$ dans la base $(e_i)$ sous la forme $G'=G+T^nX$ o\`u $G$ est la matrice de $\rho$ et $X$ une matrice \`a coefficients dans $W$.

Rappelons que $\varphi(T)=uT(p+T)^{p-1}$ avec $u$ une unit\'e de $S$; alors,
$$\begin{array}{ll}\rho^\prime(\varphi(e_j)) &=\rho^\prime((p+T)^{r_j}\sum_{i=1}^{2g}a_{ij}e_i)\\
&=\rho(\varphi(e_j))+(p+T)^{r_j}(1+T\alpha_j)T^n\sum_{1\leq i,k\leq 2g}
a_{kj}x_{ik}e_i\\ \end{array}$$
et 
$$
\begin{array}{ll}\varphi(\rho^\prime(e_j)) &=\varphi(\rho(e_j))+\varphi( T^n)\sum_{i=1}^{2g}
\varphi(x_{ij})\varphi(e_i)\\
&=\varphi(\rho(e_j))+ T^n(p+T)^{n(p-1)}u^n\sum_{1\leq i,
k\leq 2g}\varphi(x_{kj})(p+T)^{r_k}a_{ik}e_i\quad.\\ \end{array}$$
On obtient une \'equation matricielle
\`a r\'esoudre
$$(\ast)\quad\quad B=u^nA'X'_\varphi-XA,$$ o\`u $B$ est la matrice modulo $T$ du syst\`eme $(b_j)$ dans la base $(e_i)$, $A'=(p^{r_j}a_{ij})$ et $X'_\varphi=(p^{n(p-1)-r_j}\varphi(x_{ij}))$.
 
Ces matrices sont toutes \`a coefficients dans $W$, qui est complet pour la topologie $p$-adique; on commence par r\'esoudre ce syst\`eme modulo $p$.

Dans notre cas, $r_j=0$ ou $1$, donc pour $p>2$,on a $n(p-1)-r_j>0$ et l'\'equation devient:
$$B\equiv-XA\mod\ p$$
et on conclut en utilisant le fait que la matrice $A$ est inversible. 

\smallskip Notons $X_k$ une solution de $B=u^nA'X'_\varphi-XA$ modulo
$p^k$ et $B_k=B-u^nA'X'_{k,\varphi}+X_kA$. Relevons $X_k$ en $X_{k+1}=X_k+p^kY_k$ solution modulo $p^{k+1}$. L'\'equation v\'erifi\'ee par $Y_k$ est \`a nouveau $B'=-YA$ modulo $p$, qui se r\'esout en inversant la matrice $A$.
\end{proof}

\subsection{Complexit\'e}
\label{complexite} 

La d\'emonstration ci-dessus fournit un algorithme permettant d'obtenir l'action de $\Gamma$ \`a une pr\'ecision modulo $p^i$ et modulo $T^j$ souhait\'ee. 

Rappelons que calculer le produit de deux matrices carr\'ees de taille $l$ ou l'inverse d'une matrice \`a coefficients dans un anneau n\'ecesite un nombre d'op\'erations dans l'anneau d'au plus $O(l^\omega)$, o\`u $\omega$ est un r\'eel compris entre $2$ et $2,81$ (\cite{complexity_theory} chapitres 15-16). La matrice $A$ modulo $p$, dont il faut calculer l'inverse, est \'el\'ement de ${\cal M}_{2g}(k)$, d'o\`u une complexit\'e de $O(g^\omega n^{1+\varepsilon})$ pour ce calcul.

Pour effectuer la $k^{\rm i\grave eme}$ \'etape du d\'evissage modulo $p$, il faut d'abord \'elever \`a la puissance $p$ les coefficients de la matrice carr\'ee $Y_{k-1}$, ce qui se fait en $O((\log p)^{1+\varepsilon}g^{2+\varepsilon}n^{1+\varepsilon})$ op\'erations, puis calculer $B-u^nA'X'_{k,\varphi}+X_kA$, c'est-\`a-dire deux produits et sommes de matrices modulo $p^k$, calcul de complexit\'e totale de l'ordre de $O(g^\omega n^{1+\varepsilon}k^{1+\varepsilon})$.
 
Cette \'etape se r\'ep\'etant pour $k$ variant entre $1$ et $i$, la complexit\'e totale du d\'evissage modulo $p$ est $O(i(\log p)^{1+\varepsilon}g^{2+\varepsilon}n^{1+\varepsilon})+O(g^\omega n^{1+\varepsilon}i^{2+\varepsilon})$. 

Les m\^emes arguments permettent de calculer la complexit\'e du d\'evissage modulo $T$, sachant qu'\`a chaque \'etape, on effectue un d\'evissage modulo $p$. La complexit\'e totale de l'algorithme pour d\'eterminer la matrice de $\Gamma$ modulo $p^i$ et $T^j$ est donc au plus $$O(ji(\log p)^{1+\varepsilon}g^{2+\varepsilon}n^{1+\varepsilon})+O(g^\omega n^{1+\varepsilon}j^{2+\varepsilon}i^{2+\varepsilon}).$$

\section{Application au cas des courbes hyperelliptiques}
Dans cette partie, $X$ est une courbe hyperelliptique et nous reprenons les notations 
de ~\ref{coh_DR_courbes-hyperell}. 
La méthode générale précedente ~\ref{phi_Gamma_courbes} met en jeu le calcul du Frobenius 
divisé sur la cohomologie de De Rham de $X$. Nous allons donner dans la suite un
algorithme pour calculer le $(\varphi, \Gamma)$-module mod $p^i$ associé à 
$H^1_{et}(X_{\overline{K}},\mathbf{Q}_p)$ basé sur l'algorithme de Kedlaya pour 
calculer la fonction Zêta de $X_0$. Ensuite nous donnerons une autre méthode, 
basée sur le morphisme de Deligne-Illusie et notre article~\cite{huyghe-wach_interp-crist}, pour calculer 
ce $(\varphi, \Gamma)$-module modulo $p$. Enfin nous comparerons ce que nous obtenons 
en suivant l'algorithme de Kedlaya mod $p$ ou en suivant la deuxième méthode utilisant le morphisme de Deligne-Illusie.

\subsection{Calcul du $(\varphi, \Gamma)$-module associé aux courbes hyperelliptiques via l'algorithme de Kedlaya}
\label{methode_kedlaya}
La base $(({x^idx\over y})_{0\leq i \leq g-1}, (\omega'_i)_{g\leq i \leq 2g-1})$ pr\'esent\'ee dans le corollaire \ref{base} est une base adapt\'ee \`a la filtration. 
Notons $e_i={x^{i-1}dx\over y}$ (respectivement $e_i=\omega'_{i-1}$) pour $1\leq i\leq g$ (respectivement pour $g+1\leq i\leq 2g$) les \'el\'ements de cette base. On a bien, pour $1\leq i\leq g$, $$e_i\in\fil^1H^1_{DR}(X)  \quad\hbox{ et }\quad\varphi_1(e_i)={1\over p}\varphi(e_i)$$ 
et, pour $g+1\leq i\leq 2g$, $$e_i\in\fil^0H^1_{DR}(X), \quad r_i=0\quad\hbox{ et }\quad\varphi_0(e_i)=\varphi(e_i).$$ 

%Lorsque la caract\'eristique de $k$ est $2$, les \'el\'ements $\omega_i$ pour $0\leq i\leq g-1$ forment une base de $\fil^1H^1_{DR}(X)$ (cf proposition \ref{calcul_sections_globales}); comme ci-dessus, on pose $e_i=\omega_{i-1}$ pour $1\leq i\leq g$, on compl\`ete cette base et l'on obtient une matrice $A$ qui est inversible.

 Notons $X'=\spec A'$ la courbe obtenue en enlevant les points de ramification du rev\^etement $h:X\rightarrow {\bf P}^1_W$ et $X'_k$ sa fibre sp\'eciale. Kedlaya  \cite{Kedlaya_hyperell} d\'ecrit un algorithme qui permet de calculer l'action du Frobenius sur la cohomologie rigide \`a coefficients dans $K$ de la courbe $X'_k$. Cette cohomologie s'identifie \`a la cohomologie de de Rham (cf \cite{BC94}) de $X_K'$ et le $K$-espace  vectoriel $H^1_{DR}(X'_K)$ se trouve ainsi muni d'une action de Frobenius. L'inclusion $X'\hookrightarrow X$ induit des applications $H^1_{DR}(X)\rightarrow H^1_{DR}(X')$ et $H^1_{DR}(X_K)\rightarrow H^1_{DR}(X_K')$ qui sont injectives et \'equivariantes par l'action de Frobenius (cf \cite{Bogaart}). Nous pouvons donc utiliser cet algorithme pour calculer la matrice de $\varphi$ et, apr\`es changement de base dans la base form\'ee des vecteurs $(e_i)$, en d\'eduire la matrice $A$. 

Cette matrice est la donn\'ee de l'algorithme d\'ecrit en  \ref{phi-gamma-module} et nous pouvons donc en d\'eduire la matrice de $\Gamma$ \`a une pr\'ecision modulo $p^i$ et modulo $T^j$ arbitraire.  

Rapelons que la complexit\'e de l'algorithme permettant d'obtenir l'action de $\Gamma$ a \'et\'e \'evalu\'ee comme au plus $$O(ji(\log p)^{1+\varepsilon}g^{2+\varepsilon}n^{1+\varepsilon})+O(g^\omega n^{1+\varepsilon}j^{2+\varepsilon}i^{2+\varepsilon}).$$ 
Le changement de base n'a pas d'influence sur cette complexit\'e, mais il convient de rajouter la complexit\'e de l'algorithme permettant de calculer la matrice $A$ modulo $p^i$, \`a savoir $O(p^{1+\varepsilon}N^{2+\varepsilon}g^{2+\varepsilon}n^{1+\varepsilon})$ pour d\'eterminer la matrice du Frobenius modulo $p^{i+1}$ via l'algorithme de Kedlaya, l'entier $N$ v\'erifiant $N-v_p(2N+1)\geq i+1$.
Au final, la complexit\'e totale est major\'ee par $$O(g^\omega n^{1+\varepsilon}j^{2+\varepsilon}i^{2+\varepsilon}p^{1+\varepsilon}).$$

\subsection{Deuxième méthode de calcul du Frobenius divis\'e modulo $p$ pour une courbe hyperelliptique}\label{frob_divise}
\subsubsection{Rappel de la méthode}
%}
Pour un sch\'ema $Y$ sur $\spec k$, notons $Y'$ le sch\'ema d\'eduit de $Y$ par le changement de base induit par le Frobenius $\sigma:a\mapsto a^p$ sur $k$ et $F:Y\rightarrow Y'$ le Frobenius relatif.

Soit $X_0$ une courbe propre et lisse sur $\spec k$, de dimension $1$, qui se relève en une courbe propre 
et lisse $X$ sur $\spec W$. 
Sous ces conditions, la suite spectrale de Hodge vers de Rham d\'eg\'en\`ere en $E_1$ (\cite{fontaine-messing_pad_pre_et_coh}) et le choix d'un scindage de la suite exacte
$$0\rightarrow H^0(X'_0,\Omega^1_{X'_0/k})\rightarrow H^1_{dR}(X'_0/k)\rightarrow H^0(X'_0,\Omega^1_{X'_0/k})\rightarrow 0$$ d\'etermine un isomorphisme 
$$ H^0(X'_0,\Omega^1_{X'_0/k})\oplus H^1(X'_0,{\cal O}_{X'_0})\simeq H^1_{dR}(X'_0/k).$$
Par les théorèmes de comparaison de Berthelot entre la cohomologie de de Rham et la cohomologie cristalline, 
le Frobenius cristallin induit une application $$ \Phi : H^1_{DR}(X'_0)\rig H^1_{DR}(X_0).$$
Comme en ~\ref{psi-frob-divise}
 Le choix d'un scindage $H^1_{DR}(X'_0)\simeq  H^1(X'_0,\OO_{X'_0})\oplus H^0(X'_0,\Omega^1_{X'_0})$,
donne ainsi un morphisme de Frobenius divisé $$\Phi_M: H^1(X'_0,\OO_{X'_0})\oplus H^0(X'_0,\Omega^1_{X'_0})\rig H^1_{DR}(X_0),$$
qui mod $p$ ne dépend pas du scindage (4.2 de \cite{huyghe-wach_interp-crist}), 
et qui, après extension des scalaires à $\OO_{X_0}$, redonne l'application $\psi$ de ~\ref{psi-frob-divise}.

Par ailleurs, la courbe $X_0$ se rel\`eve en un schéma propre et lisse $X$ sur $\spec W$, et a posteriori mod $p^2$. 
Avec ces hypothèses d'existence d'un tel relèvement, Deligne et Illusie (\cite{deligne_illusie_dec_dr}) ont construit un quasi-isomorphisme 
$$f: \OO_{X'_0}\oplus \Omega^1_{X'_0/k}[-1]\rightarrow F_\ast\Omega^\cdot_{X_0/k}$$ qui induit l'isomorphisme de Cartier $\OO_{X'_0}$-linéaire
$$C^{-1}:\OO_{X'_0}\oplus \Omega^1_{X'_0/k}\rightarrow{\cal H}^0F_\ast\Omega^\cdot_{X_0/k}\oplus {\cal H}^1F_\ast\Omega^\cdot_{X_0/k}. $$ 
Appliquons le foncteur $R^1\Gamma(X'_0,.)$ au quasi-isomorphisme $f$ : cela donne une application 
$$R^1\Gamma(X'_0,f): H^1(X'_0,\OO_{X'_0})\oplus H^0(X'_0,\Omega^1_{X'_0})\rig H^1_{DR}(X_0).$$
L'interpr\'etation cristalline de ce morphisme (thm 4.3 de \cite{huyghe-wach_interp-crist}) affirme que $R^1\Gamma(X'_0,f)=\Phi_M$.
Ainsi, l'application de Frobenius divisé introduite 
en ~\ref{psi-frob-divise} qui permet de calculer le $(\varphi,\Gamma)$-module mod $p$ associé au $\varphi$-module 
filtré $H^1_{DR}(X)$, s'obtient par extension des scalaires à partir de l'action sur la cohomologie 
du morphisme de Deligne-Illusie. Notons au passage que la matrice $A$ introduite en ~\ref{psi-frob-divise}
ne dépend pas, mod $p$, de la base. Pour calculer cette matrice $A$ mod $p$, il nous suffit 
donc de calculer le morphisme de Deligne Illusie $f$. Etant donné un 
recouvrement affine et lisse de $X_1$ par des ouverts $(\UU_i)_{i\in I}$, l'application $f$ est un 
morphisme de complexes : $$ f : \oplus_{i=0}^1\Omega^i_{X'_0/k}[-i]\rightarrow F_\ast\check{\cal C}^.({\cal
U},\Omega^._{X_0/k}),$$ où le complexe but est le complexe simple associé au bicomplexe de \v{C}ech
du recouvrement $(\UU_i)_{i\in I}$ à valeurs dans $F_* \Omega^._{X_0/k}$. 

 Or, il s'avère que la construction de 
$f$ expliquée dans notre cas en~\ref{construction-f} est complètement explicite et 
donne un proc\'ed\'e pour obtenir algorithmiquement dans le cas d'une courbe hyperelliptique l'expression du Frobenius divis\'e modulo $p$. Plus
pr\'ecis\'ement, on calcule la matrice $A$ modulo $p$ de $\psi$ décrit en ~\ref{psi-frob-divise} dans la base $(e_i)_{1\leq i\leq 2g}$ adapt\'ee \`a la filtration. Cette m\'ethode a d\'ej\`a utilis\'ee dans le cas des courbes de Drinfeld ({\it loc. cit.}). Elle n\'ecessite le choix d'un scindage (ou d'une base adapt\'ee \`a la filtration), la connaissance d'un rel\`evement de l'endomorphisme de Frobenius modulo $p^2$ sur un recouvrement affine et utilise le complexe de \v{C}ech. 
\subsubsection{Calcul dans le cas particulier des courbes hyperelliptiques}
%}
Rappelons la situation consid\'er\'ee; $P$ est un polyn\^ome unitaire de $W[X]$ de degr\'e impair $d=2g+1$, tel que $P$ est s\'eparable comme polyn\^ome \`a
coefficients dans $K$ et tel que la classe de $P$ dans $k[X]$, not\'ee $P_1$ est aussi s\'eparable.

La courbe $X_0$ sur $k$, fibre sp\'eciale de $X$, est la r\'eunion des deux ouverts affines  
$$U_0=\spec k[x,y]/y^2-P_1(x),$$
$$V_0=\spec k[x_1,t]/t^2-Q_1(x_1),$$ 
o\`u l'on a pos\'e $Q_1(x_1)=x_1^{d+1}P_1(1/x_1)=x_1 R_1(x_1)$, $t=y/x^{g+1}$ et $x_1=1/x$.

Notons $\check{\cal C}({\cal U}, \Omega^\cdot_{X_0/k})$ le complexe simple associ\'e au bicomplexe de \v{C}ech du recouvrement ${\cal U}=(U_0,V_0)$. Le morphisme $\Omega^\cdot_{X_0/k}\rightarrow\check{\cal C}({\cal U},\Omega^\cdot_{X_0/k})$ est un quasi-isomorphisme et induit un isomorphisme
$$H^1_{dR}(X_0/k)\simeq H^1(\check{\cal C}({\cal U},\Omega^\cdot_{X_0/k})).$$

\subsubsection{Description de $H^1_{dR}(X_0/k)$}\label{description_H1}

Dans cette section, nous allons construire un scindage de la filtration de Hodge, ce qui revient \`a construire un isomorphisme 
$$H^0(X_0,\Omega^1_{X_0/k})\oplus H^1(X_0,{\cal O}_{X_0})\simeq H^1(\check{\cal C}({\cal U},\Omega^\cdot_{X_0/k}))$$
pour le recouvrement $\cal U$ de $X_0$ form\'e par $U_0$ et $V_0$.

\begin{sousprop}\label{base_gr0} Le $k$-espace vectoriel
$H^1(X_0,{\cal O}_{X_0})$ a pour base la famille de classes $[h_i]$ o\`u
$$h_i={y\over x^i}={t\over x_1^{g+1-i}}, \ \hbox{ pour }\ 1\leq i \leq g.$$
\end{sousprop}
\begin{proof}
Utilisons le recouvrement de $X_0$ par les deux ouverts affines $U_0$ et $V_0$; alors $H^1(X_0,{\cal O}_{X_0})$ s'identifie aux classes des \'el\'ements $h$ de ${\cal O}_{X_0}(U_0\cap V_0)$ qui ne s'\'etendent \`a aucun des deux ouverts $U_0$ et $V_0$, modulo la relation $h=f_U-f_V$, o\`u $f_U$ (respectivement $f_V)$ est une fonction d\'efinie sur $U_0$ (respectivement sur $V_0$). Or, d'une part 
$${\cal O}_{X_0}(U_0\cap V_0)=k\left[x,y,{1\over x}\right]=k\left[x_1,t,{1\over x_1}\right],$$
d'autre part, pour tout $i\in\bf Z$,
$${y\over x^i}={t\over x_1^{g+1-i}}.$$
On en d\'eduit que pour $1\leq i\leq g$, les sections $h_i$ ne s'\'etendent ni \`a $U_0$, ni \`a $V_0$. De plus elles forment un syst\`eme lin\'eairement ind\'ependant sur $k$.
\end{proof}

On a vu dans la proposition \ref{sect_glob_omega}
que $H^0(X,\Omega^1_{X/W})$ est le $W$-module libre de base les \'el\'ements $\omega_i=x^i{dx\over y}=x_1^{g-1-i}{dx_1\over t}$ pour $0\leq i\leq g-1$. Puisque %ref?
$H^0(X_0,\Omega^1_{X_0/k})$ s'identifie \`a la r\'eduction modulo $p$ de $H^0(X,\Omega^1_{X/W})$, c'est un $k$-espace vectoriel de base $(\bar\omega_i)_{0\leq i\leq g-1}$ r\'eduction modulo $p$ de la famille $(\omega_i)_{0\leq i\leq g-1}$.

Le groupe de cohomologie $H^1_{DR}(X_0)$ s'identifie aux classes des \'el\'ements 
$$(\omega_U,\omega_V,h)\in \Omega^1_{X_0/k}(U_0)\oplus\Omega^1_{X_0/k}(V_0)\oplus {\cal O}_{X_0}(U_0\cap V_0)$$ tels que $\omega_V-\omega_U+dh=0$, modulo les \'el\'ements de la forme $(dh_U,-dh_V,h)$ o\`u $h=h_U+h_V$.

\begin{sousprop} \label{descr-cech}\be 
\item[(i)] L'image de $\omega\in H^0(X_0,\Omega^1_{X_0/k})$ dans $H^1(\check{\cal C}({\cal U},\Omega^\cdot_{X_0/k}))$ est la classe de $(\omega_{|_U},\omega_{|_V}, 0)$.

\item[(ii)] On peut construire algorithmiquement une application $k$-lin\'eaire 
$$\begin{array}{llll} 
s: &H^1(X_0,{\cal O}_{X_0}) &\rightarrow &H^1(\check{\cal C}({\cal U},\Omega^\cdot_{X_0/k}))\\
&\hfil [h_i] \hfil&\mapsto &[(\alpha_i^U,-\alpha_i^V,h_i)]
\end{array} $$
o\`u $\alpha_i^U\in\Omega^1_{X_0/k}(U_0)$ et $\alpha_i^V\in \Omega^1_{X_0/k}(V_0)$ v\'erifiant $dh_i=\alpha_i^U + \alpha_i^V$,
qui soit un scindage de la filtration de Hodge.
\ee 
\end{sousprop}
 
\begin{proof}
Le premier point r\'esulte de l'inclusion naturelle $H^0(X_0,\Omega^1_{X_0/k})\subset H^1(\check{\cal C}({\cal U},\Omega^\cdot_{X_0/k}))$. 

Pour d\'emontrer le point $(ii)$, pour $1\leq i\leq d$, calculons
\begin{eqnarray*}
dh_i &=&-i{y^2\over x^{i+1}}{dx\over y}+{1\over x^i}\,dy\\
&=&\left({-iP_1(x)\over x^{i+1}}+{P_1'(x)\over 2x^i}\right)\,{dx\over y}\\
&=&R_i(x)\,{dx\over y}+S_i\left({1\over x}\right)\,{dx\over y} 
\end{eqnarray*}
o\`u $R_i$ et $S_i$ sont des polyn\^omes \`a coefficients dans $k$ et $\val S_i\geq1$, ce qui assure l'unicit\'e de la d\'ecomposition. On pose alors $\alpha_i^U=R_i(x){dx\over y}$ et $\alpha_i^V=S_i\left({1\over x}\right){dx\over y}$ et on d\'efinit par lin\'earit\'e une application $s$ qui \`a $[h_i]$ associe la classe de $(\alpha_i^U,-\alpha_i^V,h_i)$ dans  $H^1(\check{\cal C}({\cal U},\Omega^\cdot_{X_0/k}))$. 
\end{proof}
\begin{sousrems}\end{sousrems}
 \be 
\item[(i)] L'application ci-dessus repose sur un choix. En effet, pour une forme diff\'erentielle globale $\omega\in H^0(X_0,\Omega^1_{X_0/k})$, on peut \'ecrire
$$dh_i=\left(R_i(x)\,{dx\over y}+\omega\right)+\left(S_i\left({1\over x}\right)\,{dx\over
y}-\omega\right),$$
ce qui induit une autre section. 

\item[(ii)] On remarque que pour $1\leq i\leq g$, $\alpha_i^U$ est une section globale de $\Omega^1_{X_0}(2g\infty)$, ce qui fait le lien avec la proposition \ref{desc2_base_inf_d2} et le corollaire \ref{suite_cark_diff2}. Rappelons que $(\omega_k)_{0\leq k\leq 2g-1}$, o\`u $\omega_k=x^k{dx\over y}$ pour $0\leq k\leq 2g-1$ forme une base de $(\Omega^1_{X_0}(2g\infty))^-$; si l'on note $(b_k)_{0\leq k\leq 2g+1}$ les coefficients de $P_1$, avec $b_{2g+1}=1$, on obtient
$$\alpha_i^U=\sum_{k=i}^{2g}{1\over2}(k-2i+1)b_{k+1}\omega_{k-i}.$$

\ee

\subsubsection{Rel\`evement du Frobenius sur les ouverts affines} \label{relev-frob}

Consid\'erons $X_1$ un rel\`evement de $X_0$ modulo $p^2$ \'egalement r\'eunion de deux ouverts affines  
$$U_1=\spec W_2(k)[x,y]/y^2-P_2(x),$$ $$V_1=\spec W_2(k)[x_1,t]/t^2-Q_2(x_1),$$ o\`u $P_2$ est la r\'eduction de $P$ modulo $p^2$ et $Q_2(x_1)=x_1^{d+1}P_2(1/x_1)=x_1 R_2(x_1)$, $t=y/x^{g+1}$ et $x_1=1/x$.

Le Frobenius de $X_0$ ne se rel\`eve pas globalement sur $X_1$, mais on peut le relever sur chacun des deux ouverts affines $U_1$ et $V_1$, ce qui permet de construire le morphisme de Deligne-Illusie via le complexe de \v{C}ech.

\begin{sousprop} Il existe deux polyn\^omes $u$ et $\alpha$ \`a coefficients dans $k$ tels qu'un rel\`evement de Frobenius sur $U_1$ soit d\'etermin\'e par
\be
\item[(i)] $F_2(x)=x^p+m_p(u(x))$ et $F_2(y)=y^p+y^pm_p(\alpha(x))$;

\item[(ii)] $u$ et $\alpha$ sont reli\'es par la relation, $$(P_2(T))^p-P_2^\sigma(T^p)= m_p\bigl(P_1'(T))^pu(T)-2\alpha(T)(P_1(T))^p\bigr).$$
\ee
\end{sousprop}

\begin{proof}
La suite exacte
$$0\longrightarrow k\longrightarrow W_2(k)\longrightarrow k\longrightarrow 0,$$
dont la premi\`ere application $m_p$ est la multiplication par $p$ et la deuxi\`eme la r\'eduction modulo $p$
permet de se ramener \`a rechercher un rel\`evement du Frobenius sur $U_1$ de la forme
$$F_2(x)=x^p+m_p(u(x))\hbox{ et } F_2(y)=y^p+m_p(u(y)),$$ o\`u $u$ est une application de $k[x,y]/y^2-P_1(x)$ dans lui-m\^eme qui doit v\'erifier $$y^{2p}+2m_p(u(y))y^p=P_2^\sigma(x^p)+P'^\sigma_2(x^p)m_p(u(x)).$$
Ici, $\sigma$ d\'esigne le Frobenius sur $W_2(k)$ qui rel\`eve l'\'el\'evation \`a la puissance $p$ sur $k$ et pour un polyn\^ome $Q(T)=\sum_{k=0}^nb_kT^k$, on pose $Q^\sigma(T)=\sum_{k=0}^n\sigma(b_k)T^k$.
Le polyn\^ome $(P_2(T))^p-P_2^\sigma(T^p)$ est un multiple de $p$ dans $W_2(k)[T]$ et on l'\'ecrit $(P_2(T))^p-P_2^\sigma(T^p)=m_p({\cal P}(T))$, o\`u ${\cal P}(T)\in k[T]$. La relation v\'erifi\'ee par $u$ est alors
$$ {\cal P}(x)=P_1'^\sigma(x^p)u(x)-2u(y)y^p.$$
Recherchons $u(y)$ sous la forme $u(y)=y^p\alpha(x)$ avec $\alpha(x)\in k[x]$.
L'\'equation s'\'ecrit 
$${\cal P}(x)=P_1'^\sigma(x^p)u(x)-2\alpha(x)y^{2p}=(P_1'(x))^pu(x)-2\alpha(x)(P_1(x))^p.$$
La condition de s\'eparabilit\'e sur $P_1$ assure que $P_1$ et $P'_1$ sont premiers entre eux et le th\'eor\`eme de B\'ezout permet de conclure \`a l'existence de polyn\^omes $u$ et $\alpha$.
\end{proof}

De la m\^eme fa\c con, sur $V_1$, on obtient en rel\`evement du Frobenius $F_2(x_1)=x_1^p+m_p(v(x_1))$ et $F_2(t)=t^p+t^pm_p(\beta(x_1))$ avec la condition
$${\cal Q}(x_1)=(Q'_1(x_1))^pv(x_1)-2Q_1(x_1)^p\beta(x_1),$$
o\`u $m_p({\cal Q}(T))=(Q_2(T))^p-Q_2^\sigma(T^p)$.

\subsubsection{Frobenius divis\'e} \label{frob-div}
\label{construction-f}
Le choix d'un rel\`evement de Frobenius sur les deux ouverts affines $U_1$ et $V_1$ permet de 
construire un morphisme 
$$f:\oplus_{i=0}^1\Omega^i_{X'_0/k}[-i]\rightarrow F_\ast\check{\cal C}^.({\cal U},\Omega^._{X_0/k}),$$
en suivant pas-\`a-pas la m\'ethode \cite{deligne_illusie_dec_dr}.

En degr\'e $0$, $f:{\cal O}_{X'_0/k}\rightarrow F_\ast{\cal O}_{U_0}\oplus F_\ast{\cal O}_{V_0}$ est l'application $\sigma$-lin\'eaire de $k$-alg\`ebres qui envoie $x$ sur $x^p$.

En degr\'e $1$, $f=(\omega_U,\omega_V,h):\Omega^1_{X'_0/k}\rightarrow F_\ast\Omega^1_{U_0/k}\oplus F_\ast\Omega^1_{V_0/k}\oplus F_\ast{\cal O}_{U_0\cap V_0}$ est induite par l'application qui localement v\'erifie
\begin{eqnarray*}f_U(dx) &=&(x^{p-1}+u'(x))dx,\\
 f_V(dx_1)&=&(x_1^{p-1}+v'(x_1))dx_1\\
h(dx)&=&u(x)+x^{2p}v(x_1).\end{eqnarray*} 

Via le quasi-isomorphisme $\Omega^._{X_0'/k}\simeq\check{\cal C}^.({\cal U},\Omega^._{X'_0/k})$ et la base $\left((\bar\omega_i)_{0\leq i\leq g-1},(s([h_i]))_{1\leq i\leq g}\right)$ de $H^1_{dR}(X'_0)$, on peut calculer explicitement l'expression du Frobenius divis\'e sur la cohomologie de $X'_0$

$$\overline\Phi_M: H^0(X'_0,\Omega^1_{X'_0/k})\oplus H^1(X'_0,{\cal O}_{X'_0})\simeq H^1_{DR}(X'_0/k)\rightarrow H^1(\check{\cal C}({\cal U},F_\ast\Omega^\cdot_{X_0/k})).$$

\paragraph{Image de $\bar\omega_i$ par $\overline\Phi_M$}

Posons 
\begin{eqnarray*} f(\bar\omega_i) &=&\left(f_U(\bar\omega_i), f_V(\bar\omega_i), h(\bar\omega_i)\right)\\
&=&\left(x^{ip}\,{x^{p-1}+u'(x)\over y^{p-1}}{dx\over y},-x_1^{p(g-1-i)}{x_1^{p-1}+v'(x_1)\over t^{p-1}}{dx_1\over t},x^{ip}{u(x)+x^{2p}v(x_1)\over y^p}\right).
\end{eqnarray*}

\begin{souslem} $f(\bar\omega_i)\in F_\ast\Omega^1_{U_0/k}\oplus F_\ast\Omega^1_{V_0/k}\oplus F_\ast{\cal O}_{U_0\cap V_0}$ et v\'erifie $$f_V(\bar\omega_i)-f_U(\bar\omega_i)+d(h(\bar\omega_i))=0.$$
\end{souslem}

\begin{proof}
En effet, le polyn\^ome d\'eriv\'e de ${\cal P}$ s'exprime simplement de deux mani\`eres
$${\cal P }'(T)= (P_1(T))^{p-1}P_1'(T)-T^{p-1}(P_1'(T))^p=(P'_1(T))^pu'(T)-2\alpha'(T)(P_1(T))^p$$
et l'on en d\'eduit la relation
$$(u'(T)+T^{p-1})(P'_1(T))^p=(P_1(T))^p(P'_1(T)+2\alpha'(T)P_1(T)).$$
Puisque $P_1$ et $P'_1$ sont premiers entre eux, $(P_1(T))^p$ divise le polyn\^ome
$u'(T)+T^{p-1}$, ce qui permet de conclure que $${x^{p-1}+u'(x)\over y^{p-1}}\in{\cal
O}_{U_0}.$$

Le m\^eme raisonnement s'applique \`a la forme diff\'erentielle sur $V_0$.

Pour ce qui concerne la fonction sur $U_0\cap V_0$, on remarque que 
\begin{eqnarray*}
{\cal P}(T) &=& T^{p(d+1)}{\cal Q}\left( {1\over T}\right)\\
&=& T^{p(d+1)}\left(\left(Q_1'\left( {1\over T}\right)\right)^pv\left( {1\over T}\right)-2\beta\left( {1\over T}\right)\left(Q_1\left( {1\over T}\right)\right)^p\right)\\
&=&\bigl((d+1)^pT^p\left(P_1(T)\right)^p-T^{2p}\left(P_1'(T)\right)^p\bigr)v\left({1\over T}\right)-2\beta\left( {1\over T}\right)\left(P_1(T)\right)^p\\
&=& \left(P_1'(T)\right)u(T)-2\alpha(T)\left(P_1(T)\right)^p,\end{eqnarray*}
ce qui nous fournit une relation
$$\left(P'_1(T)\right)^p \bigl(u(T)+X^{2p}v\left({1\over T}\right)\bigr)=\left(P_1(T)\right)^p\bigl(2\alpha(T)+(d+1)^pT^pv\left( {1\over T}\right)-2\beta\left( {1\over T}\right)\bigr)$$
et permet de conclure.

\end{proof}

Explicitons \`a pr\'esent la d\'ecomposition de la classe de $f(\bar\omega_i)$ dans $H^1(\check{\cal C}({\cal U},F_\ast\Omega^\cdot_{X_0/k}))$  en fonction de la somme directe  $H^0(X_0,\Omega^1_{X_0/k})\oplus s(H^1(X_0,{\cal O}_{X_0}))= H^1_{DR}(X_0/k)$. 

Ecrivons $[(\beta_i^U,-\beta_i^V,\bar h(\bar\omega_i))]$ l'image de $[h(\bar\omega_i)]$ par $s$ (cf. proposition \ref{descr-cech}) et posons 
$$\Phi_i=(x^{ip}\,{x^{p-1}+u'(x)\over y^{p-1}}{dx\over y}-\beta_i^U,-x_1^{p(g-1-i)}{x_1^{p-1}+v'(x_1)\over t^{p-1}}{dx_1\over t}+\beta_i^V,0).$$

\begin{souscor} Avec les notations ci-dessus, $\Phi_i\in H^0(X_0,\Omega^1_{X_0/k})$ et 
$$\overline\Phi_M(\bar\omega_i)=\Phi_i+s([h(\bar\omega_i)]).$$
Cette d\'ecomposition est la d\'ecomposition adapt\'ee au scindage de la filtration de Hodge $s$.
\end{souscor}

\begin{proof} On v\'erifie que
$$x^{ip}\,{x^{p-1}+u'(x)\over y^{p-1}}{dx\over y}-\beta_i^U-(x_1^{p(g-1-i)}{x_1^{p-1}+v'(x_1)\over t^{p-1}}{dx_1\over t}+\beta_i^V)=f_U(\bar\omega_i)-f_V(\bar\omega_i)+dh(\bar\omega_i)=0;$$
une section globale de $\Omega^1_{X_0/k}$ est ainsi d\'efinie.
\end{proof}

\noindent{\bf Remarque.} Il suffit alors de d\'eterminer les composantes de $\overline\Phi_M(\bar\omega_i)$ dans la base $((\bar\omega_i)_{0\leq i\leq g-1},(s([h_i]))_{1\leq i\leq g})$ pour obtenir $g$ colonnes de la matrice.

\paragraph{Image de $[h_i]$ par $\overline\Phi_M$}

On a vu que $f(h_i)=h_i^p$ et $(0,0,h_i^p)$ est un cocycle par le fait que $d(h_i^p)=0$.
Pour $1\leq i\leq g$, \'ecrivons $s([h_i^p])=(\eta_i^U,-\eta_i^V,\bar h(h_i^p))$ et $h_i^p-\bar h(h_i)=h_U+h_V$, o\`u $h_U$ est une fonction d\'efinie sur $U$ (obtenue comme combinaison lin\'eaire \`a coefficients dans $k$ des fonctions $y\over x^i$ pour $i\leq0$ et $h_V$ une fonction d\'efinie sur $V$ (combinaison des fonctions $y\over x^i$ pour $i\geq g+1$).
Alors
$$(0,0,h_i^p)=s([h_i^p])+(-dh_U-\eta_i^U ,dh_V+\eta_i^V, 0)+(dh_U, -dh_V,h_U+h_V).$$
Posons $\Psi_i=(-dh_U-\eta_i^U ,dh_V+\eta_i^V, 0)$.

\begin{souscor} Avec les notations ci-dessus, $\Psi_i\in H^0(X_0,\Omega^1_{X_0/k})$ et 
$$\overline\Phi_M([h_i])=\Psi_i+s([h_i^p])$$
est la d\'ecomposition adapt\'ee au scindage de la filtration de Hodge $s$.
\end{souscor}
 
\noindent{\bf Remarque.} Comme ci-dessous, en d\'eterminant les composantes de $\overline\Phi_M(\bar\omega_i)$ dans la base $((\bar\omega_i)_{0\leq i\leq g-1},(s([h_i]))_{1\leq i\leq g})$, on obtient les derni\`eres $g$ colonnes de la matrice $A$ recherch\'ee.

\subsubsection{Comparaison des deux méthodes pour le calcul du $(\varphi,\Gamma)$-module mod $p$}
\label{complexite}

Notons $n$ le degr\'e de l'extension $k$ sur ${\bf F}_p$.

La   complexit\'e de l'algorithme de Kedlaya pour calculer la matrice du Frobenius sur $H^1_{DR}(X_K)$ modulo $p^2$ est $O(p^{1+\varepsilon}g^{2+\varepsilon}n^{1+\varepsilon})$  (cf \cite{Kedlaya_hyperell}). Une fois la matrice trouv\'ee, il faut effectuer encore un changement de base pour d\'eterminer la matrice $A$ dans une base du r\'eseau $H^1_{DR}(X)$. La matrice du changement de base se d\'eduit des calculs de \ref{complete_infini}: les vecteurs de la base appartenant \`a $H^0(X,\Omega^1_X)$ sont simplement multipli\'es par $2$, en revanche les autres vecteurs de la base adapt\'ee \`a la filtration sont des combinaisons lin\'eaires de tous les vecteurs de la base utilis\'ee par Kedlaya. L'op\'eration de changement de base revient alors \`a multiplier des matrices carr\'ees de taille $g$ \`a coefficients dans $k$, d'o\`u une complexit\'e en $O(g^\omega n^{1+\varepsilon})$. Au total, si l'on consid\`ere que $p$ est fix\'e, la complexit\'e de l'algorithme pour d\'eterminer la matrice $A$ via la m\'ethode de Kedlaya est $O(g^\omega n^{1+\varepsilon})$, o\`u $\omega$ est l'exposant de la multiplication des matrices et $2+\varepsilon\leq\omega$.

Evaluons maintenant la complexit\'e de l'algorithme reposant sur la m\'ethode de Deligne-Illusie. 

On commence par calculer un rel\`evement du Frobenius sur chacun des ouverts affines; pour cela, il faut utiliser deux fois l'algorithme d'Euclide \'etendu, appliqu\'e \`a des polyn\^omes de degr\'e  $\leq p(2g+1)$ \`a coefficients dans $k$. La complexit\'e est de l'ordre de $O((pgn)^{1+\varepsilon})$ (cf \cite{Bernstein}, section 22); il faut ensuite effectuer une multiplication de polyn\^omes de degr\'es au plus $(2g+1)p$ entre eux et deux divisions euclidiennes, op\'erations de complexit\'e $O((pgn)^{1+\varepsilon})$ ({\it loc. cit.} section 17). 

Pour d\'eterminer la matrice $A$ du Frobenius divis\'e via le complexe associ\'e au bicomplexe de \v Cech, on fait intervenir des listes de $2g$ \'el\'ements, repr\'esentant les composantes sur les repr\'esentants de la base. Chacune de ces listes est constitu\'ee de trois \'el\'ements, repr\'esentant la d\'ecomposition selon des classes de $$\Omega^1_{X_0/k}(U_0)\oplus\Omega^1_{X_0/k}(V_0)\oplus {\cal O}_{X_0}(U_0\cap V_0).$$ Les op\'erations qui interviennent sur ces listes sont des multiplications et additions de polyn\^omes de degr\'e $\leq p(2g+1)$ \`a coefficients dans $k$, d'une complexit\'e  de $O(n^{1+\varepsilon}g^{2+\varepsilon}p^{1+\varepsilon})$.

On en d\'eduit la complexit\'e totale qui est $O(n^{1+\varepsilon}g^{2+\varepsilon}p^{1+\varepsilon})$.

%\bibliography{biblio}
%\bibliographystyle{alpha}

\vspace{+5mm}
\noindent Christine Huyghe et Nathalie Wach \\
\noindent IRMA \\
\noindent Universit\'e Louis Pasteur \\
\noindent 7, rue René Descartes \\
\noindent 67084 STRASBOURG cedex FRANCE \\
\noindent m\'el huyghe@math.unistra.fr, wach@math.unistra.fr,
http://www-irma.u-strasbg.fr/\textasciitilde
huyghe, http://www-irma.u-strasbg.fr/\textasciitilde wach
% }}}

\end{document}